\documentclass[12pt]{article}
\usepackage{amsmath}
\usepackage{amssymb}
\usepackage[mathscr]{eucal}
\usepackage{graphicx}
\usepackage{citeref}
\usepackage{amsmath}
\usepackage{amsfonts}
\usepackage{fullpage}
\usepackage{times}
\usepackage{bbm}
\usepackage{epsfig}
\usepackage[usenames]{color}
\usepackage{hyperref}    

\def\eod{\vrule height 6pt width 5pt depth 0pt}
\newenvironment{proof}{\noindent {\bf Proof:} \hspace{.2em}}
                      {\hspace*{\fill}{\eod}}
\newcommand{\vsp}{\vskip 1em}

\newtheorem{theorem}{Theorem}
\newtheorem{lemma}[theorem]{Lemma}
\newtheorem{corollary}[theorem]{Corollary}

\newtheorem{example}[theorem]{Example}
\newtheorem{remark}[theorem]{Remark}

\newcommand{\GTS}{\mathsf{GTS}}

\newcommand{\ED}{ \mathsf{ED}}

\newcommand{\od}{ \overline{\mathrm{imm}}}
\newcommand{\odhk}{ \overline{\mathrm{hook\_imm}}}
\newcommand{\odtr}{ \overline{\mathrm{TwoRow\_imm}}}
\newcommand{\chitr}{ \mathrm{TwoRow\_\chi}}
\newcommand{\last}{ \mathrm{last}}

\newcommand{\sL}{  \mathcal{ L}}

\newcommand{\sO}{  \mathcal{ O}}

\newcommand{\hk}{ \mathrm{hook}}
\newcommand{\TwoRow}{ \mathrm{TwoRow}}
\newcommand{\RR}{ \mathbb{R}}

\newcommand{\CC}{ \mathbb{C}}
\newcommand{\NN}{ \mathbb{N}}

\newcommand{\NLP}{ \mathrm{NLP}}
\newcommand{\UHD}{ \mathrm{UHD}}
\newcommand{\RTP}{ \mathrm{GRP}}

\newcommand{\nhalf}{\lfloor n/2 \rfloor}
\newcommand{\nmhalf}{\lfloor (n-1)/2 \rfloor}

\newcommand{\fby}{\fcolorbox{blue}{yellow}}
\newcommand{\fyc}{\fcolorbox{blue}{cyan}}

\newcommand{\fya}{\fcolorbox{blue}{green}}
\newcommand{\comment}[1]{}
\newcommand{\red}[1]{\textcolor{red}{#1}}

\newcommand{\SYT}{ \mathsf{SYT}}
\newcommand{\DES}{ \mathrm{DES}}
\newcommand{\rowdiff}{ \mathrm{RowDiff}}

\newcommand{\SSS}{\mathfrak{S}}

\newcommand{\law}{\mathsf{Lexaway}}

\newcommand{\id}{\mathsf{id}}

\usepackage{mathtools}

\begin{document}
\title{Inequalities among two rowed immanants of the 
$q$-Laplacian of Trees and Odd height peaks in generalized
Dyck paths}

\author{
Arbind Kumar Lal\\
Department of Mathematics and Statistics\\
Indian Institute of Technology, Kanpur\\
Kanpur 208 016, India.\\
email: arlal@iitk.ac.in
\and
Mukesh Kumar Nagar\\
Department of Mathematics and Statistics\\
Indian Institute of Technology, Kanpur\\
Kanpur 208 016, India.\\
email: mukesh.kr.nagar@gmail.com
\and
Sivaramakrishnan Sivasubramanian\\
Department of Mathematics\\
Indian Institute of Technology, Bombay\\
Mumbai 400 076, India.\\
email: krishnan@math.iitb.ac.in
}

\maketitle

\begin{center}
{\it To the memory of Arbind Lal}
\end{center}

\begin{abstract}
Let $T$ be a tree on $n$ vertices and let $\sL_q^T$ be the $q$-analogue of its 
Laplacian.  For a partition $\lambda \vdash n$, let the normalized 
immanant of $\sL_q^T$ indexed by $\lambda$ be denoted as $\od_{\lambda}(\sL_q^T)$.
A string of inequalities among  $\od_{\lambda}(\sL_q^T)$ is known when
$\lambda$ varies over hook partitions of $n$ 
as the size of the first part of $\lambda$ decreases.  In this 
work, we show a similar sequence of inequalities when $\lambda$ varies over
two row partitions of $n$ as the size of the first part of 
$\lambda$ decreases.  Our main lemma is an identity involving
binomial coefficients and 
irreducible character values of $\SSS_n$ indexed by two row 
partitions.

Our proof can be interpreted using the combinatorics 
of Riordan paths and our main lemma admits a nice probabilisitic 
interpretation involving peaks at odd heights in generalized Dyck
paths or equivalently involving special descents in Standard 
Young Tableaux with two rows.  As a corollary, we 
also get inequalities between $\od_{\lambda_1}(\sL_q^{T_1})$ and 
$\od_{\lambda_2}(\sL_q^{T_2})$ when $T_1$ and $T_2$ are comparable
trees in the $\GTS_n$ poset and when $\lambda_1$ and $\lambda_2$
are both two rowed partitions of $n$, with $\lambda_1$ having 
a larger first part than $\lambda_2$.
\end{abstract}


\vsp


\section{Introduction}
\label{sec:intro}
Immanants of positive semi definite matrices and their generalizations have 
been a topic of interest since Schur \cite{schur-immanant-ineqs}.  Let 
$A = (a_{i,j})_{1 \leq i,j \leq n}$ be an $n \times n$  matrix.  
For a partition 
$\lambda \vdash n$, let $\chi_{\lambda}$ denote the character of the 
irreducible representation of the symmetric group 
$\SSS_n$ over $\CC$ indexed by $\lambda$.
We think of $\chi_{\lambda}: \SSS_n \mapsto \CC$ as a function from 
$\SSS_n$ to $\CC$ and for $\pi \in \SSS_n$ let $\chi_{\lambda}(\pi)$ 
denote the value of $\chi_{\lambda}$ on the permutation $\pi$.  Let
$\id \in \SSS_n$ denote the identity permutation in $\SSS_n$.
For a partition $\lambda \vdash n$, define the normalized immanant of 
$A$ as 
\begin{equation}
\label{eqn:norm_immanant_defn}
\od_{\lambda}(A) = \frac{1}{\chi_{\lambda}(\id)} \sum_{\pi \in \SSS_n} 
\chi_{\lambda}(\pi) \prod_{i=1}^n a_{i,\pi(i)}.
\end{equation}

Let $T$ be a tree on $n$ vertices and let $\sL_q^T$ be the 
$q$-analogue of its 
Laplacian (see Section \ref{sec:prelims} for definitions).  For 
$\lambda \vdash n$, let the normalized immanant of $\sL_q^T$ 
indexed by $\lambda$ be denoted as $\od_{\lambda}(\sL_q^T)$.  As done
usually, when
parts of a partition are repeated, we write such a partition with
the multiplicity as an exponent of that part.  Thus $\lambda = 1^n$ 
denotes the partition $\lambda = 1,1,\ldots,1$ with the part 
1 having multiplicity $n$.  Partitions of the form $\lambda = k,1^{n-k}$
are called hook partitions and when $n$ is clear, are denoted as
$\hk_k$.  Denote the normalized immanant of $\sL_q^T$ 
indexed by $\hk_k$ as $\odhk_k(\sL_q^T)$.  
For any tree $T$ with $n$ vertices, the following sequence of 
inequalities involving $\odhk_k(\sL_q^T)$'s was shown  by
Nagar and Sivasubramanian 
(see \cite[Lemma 27 and Theorem 2]{mukesh-siva-hook}).

\begin{theorem}[Nagar, Sivasubramanian]
\label{thm:hook_imm_ineqs}
Let $T$ be a tree on $n$ vertices with $q$-Laplacian $\sL_q^T$.
Then, for $2 \leq k \leq n$ and for all $q \in \RR$, the normalized 
immanants of $\sL_q^T$ satisfy the following:
\begin{equation}
\label{eqn:hook_imm_ineq}
\odhk_{k-1}(\sL_q^T) \leq  \odhk_k(\sL_q^T).
\end{equation}
The following version (which is stronger than 
\eqref{eqn:hook_imm_ineq} when $|q| > 1$) also holds 
\begin{equation*}
\odhk_{k-1}(\sL_q^T) + \frac{q^2-1}{k-1} \leq \frac{k-2}{k-1} \> \>  \odhk_k(\sL_q^T).
\end{equation*}
\end{theorem}


In this paper, we give a counterpart of inequality 
\eqref{eqn:hook_imm_ineq} to 
immanants indexed by partitions with at most two rows.  
For $k \geq 0$, let $\lambda = \TwoRow_k$ be the partition 
$\lambda = n-k,k$ of $n$ with at most two rows.  Since 
$\lambda$ is a partition, we must have $k \leq \nhalf$ and 
also have $k \geq 0$.  
Let $\odtr_k(\sL_q^T)$ denote the normalized immanant of $\sL_q^T$ indexed
by $\TwoRow_k$.  The main result of this paper is the following.

\begin{theorem}
\label{thm:two_row_imm_ineqs}
Let $T$ be a tree on $n \geq 5$ vertices and let $\sL_q^T$ be its $q$-Laplacian.
Then for $2 \leq k \leq \nhalf$ and for all $q \in \RR$, the normalized 
two-rowed immanants of $\sL_q^T$ satisfy the following:
\begin{equation*}
\odtr_{k-1}(\sL_q^T) 
\leq  \odtr_k(\sL_q^T).
\end{equation*}
\end{theorem}

Underlying the proof of Theorem \ref{thm:hook_imm_ineqs} 
from \cite{mukesh-siva-hook}, is a string of inequalities 
involving binomial coefficients and irreducible character 
values of $\SSS_n$ indexed by hook partitions.  Counterparts
of several relations
that are true for irreducible characters indexed by hook partitions
have been found for irreducible characters indexed by two-row
partitions.  See for example the paper by Zeilberger and Regev 
\cite{zeil-regev-surprising-relations-slc} and Bessenrodt's refinement
\cite{bessenrodt-coincidences-ejc}.
Inspired by such parallels, we get similar inequalities 
involving binomial coefficients and
irreducible character values indexed by two-row partitions in this
work (see Lemma \ref{lem:main_ineq}).  

Somewhat surprisingly, quantities that appear in the proof of Theorem 
\ref{thm:two_row_imm_ineqs} are intimately connected to 
the combinatorics of Generalized Riordan paths.  We outline the 
background and general
strategy of our proof in Section \ref{sec:prelims} and then move on to
the proof of Theorem \ref{thm:two_row_imm_ineqs} in Section
\ref{sec:main_thm}.  Our inequalities admit a probabilistic
interpretation to present which,
a little more background is given in Section 
\ref{sec:polynomials}.  
In Section \ref{sec:prob_interpret}, we give our interpretation 
involving 
peaks at odd heights in generalized Dyck paths or equivalently,
involving special descents in Standard Young Tableaux with
at most two rows.

A poset denoted $\GTS_n$ on the set of trees with $n$ vertices
was defined by Csikvari (see \cite{csikvari-poset1, csikvari-poset2}).
He showed that several tree parameters are monotonic as one goes 
up the $\GTS_n$ poset.   Results involving this poset are usually 
shown for a fixed tree parameter as the tree varies on $\GTS_n$.
Using the above results, we show that one can vary both the 
normalized two row immanant as the size of the first part 
decreases (this is the tree parameter)  {\bf and} the 
tree $T$.  In Corollary 
\ref{cor:gts_poset}, we give comparability results about  
$\odtr_{k-1}(\sL_q^{T_1})$ and $\odtr_{k}(\sL_q^{T_2})$ when
$T_1$ and $T_2$ are comparable trees in $\GTS_n$.

\section{Preliminaries}
\label{sec:prelims}

For a graph $G$ on $n$ vertices, we need the two following 
$n \times n$ matrices. Let $A$ and $D$ denote $G$'s 
adjacency matrix and the diagonal matrix with degrees 
on the diagonal respectively.  Define the $q$-Laplacian
of $G$ to be $\sL_q^G = I + (D-I)q^2 -qA$.  On setting
$q = 1$, we have $\sL_1^G = D- A$ which is the usual Laplacian
$L(G)$ of $G$.  Thus, $\sL_q^G$ is a more general matrix 
than the Laplacian and is termed the $q$-Laplacian of $G$.
The matrix $\sL_q^G$ has connections to the Ihara 
Selberg zeta function of $G$ (see Bass \cite{bass}, Foata and
Zeilberger \cite{foata-zeilberger-bass-trams}).  When the graph $G$ is 
a tree, $\sL_q^T$ is upto a scalar, the inverse of the 
exponential distance matrix $\ED_T$ of $T$ (see Bapat, 
Lal and Pati \cite{bapat-lal-pati}).  Several results about $\sL_q^T$ have
then subsequently been proved, see 
\cite{bapat06-procICDM,bapat-siva-third-immanant}  and the references 
therein.  Thus, $\sL_q^T$ is a well studied object.

Normalized immanants of the $q$-Laplacian $\sL_q^T$ of a
tree $T$ can be 
computed using the dual and alternative notion of vertex 
orientations.  We refer the reader to 
\cite[Lemmas 5, 17 and Theorem 11]{mukesh-siva-hook} for an 
introduction to this and terms undefined here.   We have not 
defined them again in this paper as we do not have anything 
new to say on them.
For $\lambda \vdash n$ and $j \leq \nhalf$, 
denote by $\chi_{\lambda}(j)$, the irreducible character 
$\chi_{\lambda}$ evaluated at a permutation with cycle type
$2^j1^{n-2j}$.   

For a tree $T$, when $i \geq 1$,
let $\sO_i$ denote the set of vertex orientations 
with  $i$ bidirected arcs.   
We need the following from \cite[Corollary 13]{mukesh-siva-hook}.
There exists a statistic $\law: \sO_i \mapsto \NN$ whose ordinary
generating function $a_i^T(q) = \sum_{O \in \sO_i} q^{\law(O)}$
will be used to compute normalized immanants of $\sL_q^T$.
In this work, we will need the tree $T$ and so we have embedded it
in our notation of $a_i^T(q)$.  This was not needed in
\cite{mukesh-siva-hook} and there the same quantity was denoted
$a_i(q)$ as the tree $T$ was implicit.  With this slight change
in notation, we recall  \cite[Lemma 17]{mukesh-siva-hook} 
as a starting point of this work.  

\begin{lemma}[Nagar and Sivasubramanian]
\label{lem:prelim_part}
Let $T$ be a tree on $n$ vertices with $q$-Laplacian $\sL_q^T$.  For
$\lambda \vdash n$, let $\od_{\lambda}(\sL_q^T)$ denote the 
normalized immanants of $\sL_q^T$ indexed by $\lambda$.  Then, we have
\begin{eqnarray}
\od_{\lambda}(\sL_q^T) 
 & = & \frac{1}{\chi_{\lambda}(\id)} \sum_{i=0}^{\nhalf} 
a_i^T(q) \Bigg( 
\sum_{j=0}^{i} \binom{i}{j} \chi_{\lambda}(j) \Bigg), \\
& = & \frac{1}{\chi_{\lambda}(\id)} \sum_{i=0}^{\nhalf} a_i^T(q) 2^i \alpha_{n,\lambda,i}, \label{eqn:earlier_imp}
\end{eqnarray}
where we define
\begin{equation}
\label{eqn:defn_alpha}
2^i\alpha_{n,\lambda,i} = \sum_{j=0}^{i} \binom{i}{j} \chi_{\lambda}(j).
\end{equation}
\end{lemma}


\subsection{Two rowed immanants}

Fixing the number $n$ of vertices of $T$, we specialise Lemma 
\ref{lem:prelim_part} to the case when
$\lambda = \TwoRow_k$ and denote the irreducible 
character indexed by $\TwoRow_k$ as 
$\chitr_{n,k}(\cdot )$.  
We denote the normalized immanant of $\sL_q^T$ corresponding to the
partition $\TwoRow_k$ 
as $\odtr_k(\sL_q^T)$.  When $\lambda = \TwoRow_k$,
we also denote the term $\alpha_{n,\lambda,i}$ as 
$\alpha_{n,k,i}$ to avoid any confusion.  With this notation, 
substituting $\lambda = \TwoRow_k$ in Lemma \ref{lem:prelim_part}, 
gives us:
\begin{equation}
  \label{eqn:main}
  \odtr_k(\sL_q^T) = \frac{1}{\chitr_{n,k}(\id)} \sum_{i=0}^{\nhalf} a_i^T(q) 2^i\alpha_{n,k,i}.
\end{equation}

Equation \eqref{eqn:earlier_imp} and hence equation
\eqref{eqn:main} shows that one can split the computation of 
the normalized immanant of $\sL_q^T$ into two parts: the first 
being $a_i^T(q)$ which only depends on the tree $T$ and does not
depend on the partition $\lambda$ (or $\TwoRow_k$).  
The second part is $2^i \alpha_{n,\lambda,i}$ 
(or $2^i \alpha_{n,k,i}$) which by Lemma 
\ref{lem:prelim_part}, depends on the character values of
$\lambda$ (or $\TwoRow_k$) and does not depend on the tree $T$.

Chan and Lam in \cite{chan-lam-binom-coeffs-char} 
showed that $\alpha_{n,\lambda,i} \geq 0$ for all $n,i$ and 
$\lambda \vdash n$.  
As mentioned above, when $i \geq 1$, the $a_i^T(q)$'s are 
ordinary generating functions of the statistic $\law$ and hence
have non negative integral coefficients.  Further, by 
\cite[Lemma 16]{mukesh-siva-hook} when $i \geq 1$, 
$a_i^T(q)$ is actually a polynomial in $q^2$ with 
non negative coefficients 
while $a_0(q) = 1-q^2$ for all trees
(see \cite[Corollary 13]{mukesh-siva-hook}).  
We wish to compare $\odtr_k(\sL_q^T)$ with $\odtr_{k+1}(\sL_q^T)$ 
for the same tree $T$.  Thus, we will analyse $\alpha_{n,k,i}$
in more detail. 



\section{Proof of Theorem \ref{thm:two_row_imm_ineqs}}
\label{sec:proof_main_theorem}
Our first lemma gives a recurrence between the $\alpha_{n,k,i}$'s.  
We adopt the convention that
$\alpha_{n,k,i} = 0$ if either $i > \nhalf$ or if $k > \nhalf$.

\begin{lemma}
\label{lem:murn-naka-based-recur}
For positive integers $n \geq 3$, when $i \leq  \nmhalf$, we have 
$\alpha_{n,k,i} = \alpha_{n-1,k,i} + \alpha_{n-1,k-1,i}$.
\end{lemma}
\begin{proof}
When $n=3$, the relation is easy to check.  Thus, let $n \geq 4$.  
By definition, we have
$\displaystyle 2^i \alpha_{n,k,i} = \sum_{j=0}^i \chitr_{n,k}(j) \binom{i}{j}$.
By the Murnaghan-Nakayama lemma (see Sagan's book 
\cite[Theorem 4.10.2]{sagan-book}), we have
$\chitr_{n,k}(j) = \chitr_{n-1,k}(j) + \chitr_{n-1,k-1}(j)$.  Thus, we get 
\begin{eqnarray*}
\alpha_{n,k,i} 
& = & \frac{1}{2^i} \sum_{j=0}^i \chitr_{n,k}(j)  \binom{i}{j}   \\
& = & \frac{1}{2^i} \sum_{j=0}^i \Big[ \chitr_{n-1,k}(j) + \chitr_{n-1,k-1}(j) \Big] \binom{i}{j}    \\
& = & \alpha_{n-1,k,i} + \alpha_{n-1,k-1,i}.
\end{eqnarray*}
The proof is complete.
\end{proof}

\begin{example}
\label{eg:recursive-from-prev}
We illustrate Lemma \ref{lem:murn-naka-based-recur} when
$n=6,7,8$  below, where we show the 
tables containing $\alpha_{n,k,i}$'s.

\vspace{3 mm}
\noindent
$\begin{array}{l  r}
 \begin{array}{r|c|c|c|c|} 
  & \lambda = 6 & \lambda = 5,1 & \lambda = 4,2 & \lambda = 3,3 \\ \hline
  i=0 & \red{1} & \red{5} & \red{9} & \red{5} \\ \hline
  i=1 & 1 & \fya{4} & \fya{6} & 3 \\ \hline
  i=2 & 1 & 3 & 4 & 2 \\ \hline
  i=3 & 1 & 2 & 3 & 1 \\ \hline
  \end{array}
& 
\begin{array}{r|c|c|c|c|} 
  & \lambda = 7 & \lambda = 6,1 & \lambda = 5,2 & \lambda = 4,3 \\ \hline
  i=0 & \red{1} & \fby{\red{6}} & \fby{\red{14}} & \red{14} \\ \hline
  i=1 & 1 & 5 & \fya{10} & 9 \\ \hline
  i=2 & 1 & 4 & \fyc{7} & \fyc{6} \\ \hline
  i=3 & 1 & 3 & 5 & 4 \\ \hline
  \end{array}
\end{array}$

\vspace{3 mm}

$ \begin{array}{r|c|c|c|c|c|c|} 
  & \lambda = 8 & \lambda = 7,1 & \lambda = 6,2 & \lambda = 5,3 & \lambda=4,4\\ \hline
  i=0 & \red{1} & \red{7} & \fby{\red{20}} & \red{28} & \red{14} \\ \hline
  i=1 & 1 & 6 & 15 & 19 & 9\\ \hline
  i=2 & 1 & 5 & 11 & \fyc{13} & 6\\ \hline
  i=3 & 1 & 4 & 8 & 9 & 4 \\ \hline
  i=4 & 1 & 3 & 6 & 6 & 3 \\ \hline
\end{array}$

\vspace{3 mm}

We have coloured the cells to illustrate Lemma
\ref{lem:murn-naka-based-recur}.  The coloured cell 
in the table when $n=8$ is the sum of the two identically
coloured cells in the table when $n=7$ and similarly each 
cell can be computed recursively.
\end{example}

Since $0 \leq i \leq \nhalf$, we get one extra row (that is, one 
extra $i$) for each even $n$ as $n$ increases.  When $n=2 \ell$
is even, we denote the row corresponding to $i = \ell$ as
the {\it last row.}  As can be seen from Example 
\ref{eg:recursive-from-prev}, when $n=2\ell$, the 
{\it last row} cannot be obtained as a sum of rows 
when $n=2\ell-1$.

\begin{remark}
\label{rem:dim_i-zero-row}
In \eqref{eqn:defn_alpha}, with $\lambda = \TwoRow_k$, 
the dimension of the irreducible 
representation indexed by $\TwoRow_k$ equals $\alpha_{n,k,0}$.  
That is, $\alpha_{n,k,0} = \chitr_{n,k}(\id)$.   Further, 
by the Hook-length formula (see Sagan \cite{sagan-book})  the 
dimension of the irreducible representation indexed 
by $\TwoRow_k$  is $\alpha_{n,k,0} = \binom{n}{k} - \binom{n}{k-1}$.
\end{remark}

To prove Theorem \ref{thm:two_row_imm_ineqs}, we write 
both $\odtr_k(\sL_q^T)$ and $\odtr_{k+1}(\sL_q^T)$ 
using 
\eqref{eqn:main}.  As the terms $2^ia_i^T(q)$ are common
within the summation, using Remark \ref{rem:dim_i-zero-row}, 
as a first attempt, we want to show that
\begin{equation}
\label{eqn:to_show}
\frac{\alpha_{n,k,i}}{\alpha_{n,k,0}} 
\geq 
\frac{\alpha_{n,k+1,i}}{\alpha_{n,k+1,0}}.
\end{equation}

If this holds, then we can show an inequality for
each term of the summation in \eqref{eqn:main}. 
We check with our data for $n=6$ and tabulate the ratio 
$\displaystyle \frac{\alpha_{6,k,i} }{\alpha_{6,k,0}}$.  Remark
\ref{rem:dim_i-zero-row} tells us that $\chi_k(\id)$ is given 
by the first row (shown in red colour, seen better on a colour 
monitor).  Dividing, we get the following table of ratios.

\vspace{3 mm}

$ \displaystyle \begin{array}{r|c|c|c|c|} 
  & \lambda = 6 & \lambda = 5,1 & \lambda = 4,2 & \lambda = 3,3 \\ \hline
  i=0 & \red{1} & \red{5/5} & \red{9/9} & \red{5} \\ \hline
  i=1 & 1 & 4/5 & 6/9 & 3/5 \\ \hline
  i=2 & 1 & 3/5 & 4/9 & 2/5 \\ \hline
  i=3 & 1 & 2/5 & 3/9 & 1/5 \\ \hline
  \end{array}$
\vspace{1 mm}

As the data seems to agree with \eqref{eqn:to_show}, we 
will prove this first.  We need 
the following lemma whose proof is easy and hence omitted.
\begin{lemma}
\label{lem:ineq}
For positive real numbers $a,b,c,d$, 
$$\min\left( \frac{a}{c}, \frac{b}{d} \right) \leq \frac{a+b}{c+d} 
\leq \max\left( \frac{a}{c}, \frac{b}{d} \right).$$
\end{lemma}

Our next lemma shows that \eqref{eqn:to_show} holds except for
the last row.  

\begin{lemma}
\label{lem:inequality}
For positive integers $n \geq 2$ and integers $k \leq \nhalf$,
$i < \nhalf$, we have
\begin{equation}
\label{eqn:ratio-proved}
\frac{\alpha_{n,k,i}}{\alpha_{n,k,0}} \geq 
\frac{\alpha_{n,k+1,i} }{ \alpha_{n,k+1,0}  }
\end{equation}

\end{lemma}
\begin{proof}
We use induction on $n$.  The base case when $n=2$ can be 
easily checked.  For $n+1$ with $k \leq \nhalf$ and $i < \nhalf$, 
we will show that 
\begin{equation*}
\frac{\alpha_{n+1,k,i}}{\alpha_{n+1,k,0} } 
 \geq 
\frac{\alpha_{n+1,k+1,i} }{ \alpha_{n+1,k+1,0} }.
\end{equation*}
By Lemma \ref{lem:murn-naka-based-recur}, we need to show 
\begin{equation*}
\frac{\alpha_{n,k,i} + \alpha_{n,k-1,i} }{\alpha_{n,k,0} + \alpha_{n,k-1,0} } 
\geq 
\frac{\alpha_{n,k+1,i} + \alpha_{n,k,i} }{ \alpha_{n,k+1,0} + \alpha_{n,k,0} }.
\end{equation*}
As $\alpha_{n,k,i} > 0$, by 
Lemma \ref{lem:ineq}, we have
\begin{equation}
  \label{eqn:min}
\frac{\alpha_{n,k,i} + \alpha_{n,k-1,i} }{ \alpha_{n,k,0} + \alpha_{n,k-1,0}} 
\geq \min\left( \frac{ \alpha_{n,k,i} }{ \alpha_{n,k,0} }, 
\frac{ \alpha_{n,k-1,i} }{ \alpha_{n,k-1,0}  } \right) = 
\frac{\alpha_{n,k,i} }{ \alpha_{n,k,0 } }.
\end{equation}
By induction, as $\displaystyle \frac{\alpha_{n,k,i}}{\alpha_{n,k,0}} \geq 
\frac{\alpha_{n,k+1,i}}{\alpha_{n,k+1,0}}$, we also get
\begin{equation}
  \label{eqn:max}
  \frac{ \alpha_{n,k,i} }{ \alpha_{n,k,0} } =
  \max\left( \frac{ \alpha_{n,k+1,i} }{ \alpha_{n,k+1,0} }, 
  \frac{ \alpha_{n,k,i} }{ \alpha_{n,k,0} } \right) \geq 
\frac{\alpha_{n,k+1,i} + \alpha_{n,k,i} }{\alpha_{n,k+1,0} + \alpha_{n,k,0} }.
\end{equation}
Equations \eqref{eqn:min} and \eqref{eqn:max} imply that
\begin{equation}
\frac{\alpha_{n+1,k,i}}{\alpha_{n+1,k,0} } =
\frac{ \alpha_{n,k-1,i} + \alpha_{n,k,i} }
{ \alpha_{n,k-1,0} + \alpha_{n,k,0} } 
\geq 
\frac{\alpha_{n,k,i} }{ \alpha_{n,k,0 } }
\geq 
\frac{\alpha_{n,k,i} + \alpha_{n,k+1,i} }{\alpha_{n,k,0} + \alpha_{n,k+1,0}}  
= \frac{\alpha_{n+1,k+1,i} }{ \alpha_{n+1,k+1,0} }.
\end{equation}
The proof is complete.
\end{proof}

Lemma \ref{lem:inequality} shows that when $i < \nhalf$, we get 
inequality \eqref{eqn:to_show}.  When $n=2 \ell$, Lemma 
\ref{lem:inequality} does not work for the {\it last row}, that
is when $i = \ell$.  
Our next task is to extend Lemma \ref{lem:inequality} 
to this case.





\subsection{When $n=2\ell$ and $i=\ell$}
\label{sec:main_thm}

In this subsection, we will focus on extending 
Lemma \ref{lem:inequality} to the case when 
$n= 2\ell$ and when $i = \ell$.  Our first lemma
shows that the $\alpha_{2\ell,k,\ell}$'s are differences 
of succesive coefficients of the trinomial coefficients. 
Let  $p_{\ell,k}$ denote the coefficient of $x^k$ in 
$(1+x+x^2)^{\ell}$ and to save one subscript, let 
$\last_{\ell,k} =\alpha_{2\ell,k,\ell}$.
Thus, $\last_{\ell,k}=0$ when $k>\ell$ and when $k < 0.$

 \begin{lemma}
\label{lem:ank:rec:1}
Fix positive integers $\ell, k$ with $\ell \geq 2$ and with 
$0 \leq k \leq \ell$.  Then, we have
\begin{enumerate}
\item  \label{lem:ank:rec:1:2} 
$\last_{\ell,k} = \last_{\ell-1,k} + \last_{\ell-1,k-1} + \last_{\ell-1, k-2}$. 
\item  \label{lem:ank:rec:1:3} 
$p_{\ell,k} = p_{\ell-1,k} + p_{\ell-1,k-1}+p_{\ell-1,k-2}.$
\item  \label{lem:ank:rec:1:4} 
$\alpha_{2 \ell, k, \ell} =  \last_{\ell,k} = p_{\ell,k} - p_{\ell,k-1}$ is a  positive integer.
\end{enumerate}
 \end{lemma}
 \begin{proof}
 	\begin{enumerate}
\item This follows from the Murnaghan Nakayama lemma, see
\cite[Lemma 2.1]{chan-lam-binom-coeffs-char}.
\item  By the definition of $p_{\ell,k}$, we have 
\begin{align*}
p_{\ell,k} &= \mbox{ Coeff. of } x^k \mbox{ in } (1+x+x^2)^{\ell}\\
&= \mbox{ Coeff. of } x^k \mbox{ in } (1+x+x^2)^{\ell-1}(1+x+x^2)=p_{\ell-1,k} 
+ p_{\ell-1,k-1}+p_{\ell-1,k-2}.
\end{align*}
\item We induct on $\ell$.  When  $\ell=2,3$, it is easy to see that 
$\last_{\ell,k} = p_{\ell,k} - p_{\ell,k-1}$ is a  positive 
integer (also see Table \ref{tab:last_l,k}). Let the 
result be true for all positive integers less than $\ell$. From part (1) 
we have  
\begin{align*}
\last_{\ell,k} & = \last_{\ell-1,k} + \last_{\ell-1,k-1} + \last_{\ell-1, k-2} \\
& = p_{\ell-1,k} - p_{\ell-1,k-1}+p_{\ell-1,k-1} - p_{\ell-1,k-2}
+p_{\ell-1,k-2} - p_{\ell-1,k-3} \\
& = p_{\ell-1,k} +p_{\ell-1,k-1}+p_{\ell-1,k-2}- p_{\ell-1,k-1} - 
p_{\ell-1,k-2} - p_{\ell-1,k-3} \\
& = p_{\ell,k}-p_{\ell,k-1}.
\end{align*}
\end{enumerate}
The proof is complete.
\end{proof}

We need a couple of inequalities which we see in the next few
lemmas.  Our  next lemma is an extension of Lemma \ref{lem:ineq}.

\begin{lemma}
\label{lem:frac:3}
For $1 \le i \le 4$, let $a_i, b_i$ be positive integers with 
$\dfrac{a_1}{b_1} \le \dfrac{a_2}{b_2} \le \dfrac{a_3}{b_3} 
\le \dfrac{a_4}{b_4}$. Then, 
$$\dfrac{a_1 + a_2 + a_3}{b_1 + b_2 + b_3} \le 
\dfrac{ a_2 + a_3+a_4}{ b_2 + b_3+b_4}.$$ 
\end{lemma}
\begin{proof}
We know  
$\dfrac{a_1}{b_1} \le \dfrac{a_2}{b_2} \le 
\dfrac{a_3}{b_3} \le \dfrac{a_4}{b_4}$.  Lemma \ref{lem:ineq}
implies that 
$\dfrac{a_1}{b_1} \le \dfrac{a_2}{b_2} \le 
\dfrac{a_2+a_3}{b_2+b_3} \le \dfrac{a_3}{b_3} \le \dfrac{a_4}{b_4}$. 
Applying Lemma \ref{lem:ineq} again, we get 
$$\dfrac{a_1}{b_1} \le   \dfrac{a_1+ a_2+a_3}{b_1+ b_2+b_3} 
\le \dfrac{a_2 + a_3}{b_2 + b_3} \le 
\dfrac{ a_2 + a_3+a_4}{ b_2 + b_3+b_4} \le  \dfrac{a_4}{b_4}, \mbox{ completing the proof.}$$ 
\end{proof}
 
We next compare the ratio $\last_{\ell,k+1}/\last_{\ell-1,k}$
as $k$ decreases.

\begin{lemma}
\label{lem:ank:rec:2}
Fix positive integers $\ell, k $ with $1 \le k \le \ell-1$ 
and $\ell \geq 3$. Then, 
$$\dfrac{\last_{\ell,k+1}}{\last_{\ell-1,k} } \le \dfrac{\last_{\ell,k}}{\last_{\ell-1,k-1} }.$$ 
\end{lemma}
\begin{proof}
We use induction on $\ell$.  The result is easily verified 
when $\ell = 3,4$ (see Table \ref{tab:last_l,k} in 
Example \ref{example:tab_last} below). 
Let the result be true when $N \le \ell$ and when $ 1 \le k \le \ell-1$.
We then show that it holds for $\ell+1$. Thus, we need to show that 
$\dfrac{\last_{\ell+1,k+1}}{\last_{\ell,k} } \le 
\dfrac{\last_{\ell+1,k}}{\last_{\ell,k-1} }$.  By Lemma~\ref{lem:ank:rec:1}, 
this is equivalent to showing that 
$$\dfrac{\last_{\ell,k+1} + \last_{\ell,k} + \last_{\ell, k-1}}
{\last_{\ell-1,k} + \last_{\ell-1,k-1} + \last_{\ell-1, k-2} }  
\le 
\dfrac{\last_{\ell,k} + \last_{\ell,k-1} + \last_{\ell, k-2}}
{\last_{\ell-1,k-1} + \last_{\ell-1,k-2} + \last_{\ell-1, k-3} }.$$
By the induction hypothesis, we know that 
$$\dfrac{\last_{\ell,k+1}}{\last_{\ell-1,k} } \le 
\dfrac{\last_{\ell,k}}{\last_{\ell-1,k-1} } \le  
\dfrac{\last_{\ell,k-1}}{\last_{\ell-1,k-2} } \le  
\dfrac{\last_{\ell,k-2}}{\last_{\ell-1,k-3} }.$$
The proof is complete by applying Lemma~\ref{lem:frac:3}.
 \end{proof}

By Lemma \ref{lem:ank:rec:1}, 
$\last_{\ell,k}$ is a positive integer. 
Rearranging terms a bit, as an immediate consequence of 
Lemma~\ref{lem:ank:rec:2}, we obtain the following corollary.
\begin{corollary}
\label{cor:ank:rec}
Fix positive integers $\ell, k $ with $ 1 \le k \le \ell-1$ 
and with $\ell\geq 3$. Then, we have
$$\dfrac{\last_{\ell,k+1}}{\last_{\ell,k} } \le \dfrac{\last_{\ell-1,k}}{\last_{\ell-1,k-1} }.$$ 
In general, we have
$$\dfrac{\last_{\ell,k+1}}{\last_{\ell,k}} \le 
\dfrac{\last_{\ell - r,k+1-r}}{\last_{\ell - r, k - r}}, 
\mbox{ when } 1 \le r \le k.$$
\end{corollary}

\begin{example}
	\label{example:tab_last}
	We illustrate Lemma \ref{lem:ank:rec:2} and Corollary 
	\ref{cor:ank:rec} when
	$\ell,k\in\{0,1,\ldots,9\}$  in Table 
\ref{tab:last_l,k} which contains $\last_{\ell,k}$'s. 
From the table, one can verify 
	Lemma \ref{lem:ank:rec:2} when $3\leq \ell\leq 9$ and can also check 
	when $\ell=2$ and $k=1$ that we have $\last_{\ell-1,k}=0$.
	Thus, Lemma \ref{lem:ank:rec:2} is not true when $\ell=2$.
	\begin{table}
		$ \displaystyle \begin{array}{r|c|c|c|c|c|c|c|c|c|c|c|} 
			& k=0  & k=1  &  k=2  & k=3  & k=4   & k=5  & k=6   & k=7   & k=8  & k=9    \\ \hline
			\ell=0  &  1 &    &   &    &   &    &   &   &   &      \\  \hline
			\ell=1    &  1 &   0 &   &    &   &    &   &   &    &   \\  \hline
			\ell=2      & 1  &  1  &  1 &    &   &    &   &     &   &   \\  \hline
			\ell=3        &1   &  2  &  3 &  1  &   &    &     &   &   &   \\  \hline  
			\ell=4          & 1  &   3 & 6  &  6  &  3 &      &  &   &   &   \\   \hline
			\ell=5            &1   &   4 & \fya{10}  & \fby{15}   & \fyc{15}     & 6  &   &   &   &   \\   \hline
			\ell=6             &  1 &  5  &  15 &  \fya{29}    &  \fby{40}  & \fyc{36}  & 15  &   &   &   \\   \hline
			\ell=7               & 1  & 6   & 21    &49   & 84   & 105  & 91  & 36  &   &   \\   \hline
			\ell=8                  &1   &   7  &  28   & 76  &  154  & 238  & 280  & 232  & 91  &   \\  \hline
			\ell=9                   &1    &  8  & 36   & 111  & 258   & 468  & 672  & 750  & 603  & 232  \\
			\hline
		\end{array}$
		\caption{The values of $\last_{\ell,k}$.}
		\label{tab:last_l,k}
		
	\end{table}
\end{example}

With this preparation, we can prove the following result.
 
\begin{lemma}
\label{lem:ank:rec}
Fix positive integers $\ell, k $ with $ 1 \le k \le \ell-1$ and 
with $\ell\geq 3$. Then, we have
\begin{equation}\label{eqn:ank:rec}
\dfrac{ \last_{\ell, k+1} }{ {{2\ell} \choose {k+1}} - {{2\ell} \choose {k}}} \le 
\dfrac{ \last_{\ell,k} }{ {{2\ell} \choose {k}} - {{2\ell} \choose {k-1}} }.
\end{equation}
\end{lemma}
\begin{proof}
Note that ${{2\ell} \choose {k}} $ equals the coefficient of 
$x^k$ in the expansion of $$(1+x)^{2\ell} =   (1 + x + x^2 + x)^{\ell} = 
\sum_{r=0}^{\ell} {\ell \choose r} x^r (1 + x + x^2)^{\ell- r}.$$
Thus, ${ {2 \ell} \choose {k}} = \sum_{r=0}^{\ell} {\ell \choose r} p_{\ell-r, k-r},$
and hence
\begin{eqnarray*}
 {{2\ell} \choose {k+1}} - {{2\ell} \choose {k}} &=& 
\sum_{r=0}^{\ell} {\ell \choose r} p_{\ell-r, k+1-r} - 
\sum_{r=0}^{\ell} {\ell \choose r} p_{\ell-r, k-r} \\
 &=& \sum_{r=0}^{\ell} {\ell \choose r} \last_{\ell-r, k+1-r}  
 = \last_{\ell,k+1} + \sum_{r=1}^{\ell}{\ell \choose r} 
\last_{\ell-r, k+1-r}.
\end{eqnarray*}
 Hence, 
\begin{equation}
\label{eqn:pnk+1}
\dfrac{{{2\ell} \choose {k+1}} - {{2\ell} \choose {k}}}{\last_{\ell,k+1} }  = 1 + \sum_{r=1}^{\ell} {\ell \choose r} \dfrac{ \last_{\ell-r, k+1-r} }{ \last_{\ell,k+1} }.
\end{equation} 
 Similarly, 
\begin{equation}\label{eqn:pnk} 
\dfrac{{{2\ell} \choose {k}} - {{2\ell} \choose {k-1}}}{\last_{\ell,k}}  
= 
1 + \sum_{r=1}^{\ell} 
{\ell \choose r} \dfrac{ \last_{\ell-r, k-r} }{\last_{\ell,k} }.
\end{equation}
Thus, to obtain our required result, we need to show that 
$$\dfrac{\last_{\ell-r, k+1-r}  }{\last_{\ell,k+1}} \ge 
\dfrac{ \last_{\ell-r, k-r} }{\last_{\ell,k} } 
\Leftrightarrow 
\dfrac{\last_{\ell,k+1} }{ \last_{\ell,k} } \le 
\dfrac{\last_{\ell-r, k+1-r} }{\last_{\ell-r, k-r}  }.$$
By Corollary~\ref{cor:ank:rec}, 
$\dfrac{\last_{\ell,k+1}}{\last_{\ell,k}} \le 
\dfrac{\last_{n - r,k+1-r}}{\last_{n - r, k - r}}$ 
for $1 \le r \le k$, completing the proof.
\end{proof}


More generally, for positive integers $r,s$, let $a_{\ell,k,s}$ and 
$a_{\ell,k,r+s}$ denote the coefficient of $x^k$ in $(1+sx+x^2)^{\ell}$ and
in $(1+(r+s)x+x^2)^{\ell}$ respectively.  The proof of Lemma
\ref{lem:ank:rec} actually shows the following more general result.
Since the proof is identical, we only mention its statement.

\begin{remark}
\label{rem:more_general_ineq}
With the above notation for $a_{\ell,k,s}$, the following inequality is true
for all positive integers $\ell$ and all $k$.
\begin{equation}
\label{eqn:more_gen_ieq}
\frac{a_{\ell,k+1,s} - a_{\ell,k,s} }{a_{\ell,k+1,r+s} - a_{\ell,k,r+s} } 
\geq \frac{a_{\ell,k,s} - a_{\ell,k-1,s} }{ a_{\ell,k,r+s} - a_{\ell,k-1,r+s}}.
\end{equation}
Indeed, Lemma \ref{lem:ank:rec} is a special case of 
this result with $r=s=1$.
\end{remark}

As Lemma \ref{lem:ank:rec} extends Lemma \ref{lem:inequality} to 
the case when $n=2\ell$ and $i=\ell$, we record this formally
below.
\begin{lemma}
\label{lem:main_ineq}
With $\alpha_{n,k,i}$ as defined in \eqref{eqn:defn_alpha},
for positive integers $n \geq 5$ and integers $k,i \leq \nhalf$,
we have
\begin{equation}
\label{eqn:ratio-proved-fully}
\frac{\alpha_{n,k,i}}{\alpha_{n,k,0}} \geq 
\frac{\alpha_{n,k+1,i} }{ \alpha_{n,k+1,0}}.
\end{equation}
\end{lemma}
\begin{proof}
We induct on $n$ with the base case being $n=5$.  When $n=5$, the 
inequality is easy to verify and thus we can assume $n \geq 6$.
If $n$ is odd, then by Lemma \ref{lem:inequality}, we are done.
If $n = 2\ell$ is even, then Lemma \ref{lem:inequality} shows the 
inequality for $i$ from $0$ to $\ell-1$.  
Lemma \ref{lem:ank:rec} shows the inequality when $i=\ell$, completing
the proof.
\end{proof}

\begin{remark}
\label{rem:why_5}
We mention our reason as to why Lemma \ref{lem:main_ineq}
requires $n \geq 5$.
On $n=4$ vertices, there are two trees: the path tree $P_4$ and the star 
tree $S_4$. It is very easy to check that 
Theorem \ref{thm:two_row_imm_ineqs} is true for $S_4$.
However, when $T=P_4$, we have 
$\odtr_1(\sL_q^T)=1+3q^2+\frac{4}{3}q^4$ and  
$\odtr_{2}(\sL_q^T)=1+2q^2+2q^4 $. Hence,  when $|q|$ is sufficiently
large, the 
inequality given in Theorem \ref{thm:two_row_imm_ineqs} is not true 
for $P_4$ when  $k=2.$  This however is the only aberration.
\end{remark}

We are now in a position to prove Theorem \ref{thm:two_row_imm_ineqs}.

\noindent
{\bf Proof of Theorem  \ref{thm:two_row_imm_ineqs}:}
By \eqref{eqn:main} and Remark \ref{rem:dim_i-zero-row} we have, 
\begin{eqnarray*}
\odtr_k(\sL_q^T) & = & \sum_{i=0}^{\nhalf}a_i^T(q)2^i 
\frac{\alpha_{n,k,i}}{\alpha_{n,k,0}} \mbox{ and }\\
\odtr_{k+1}(\sL_q^T)  & =  & \sum_{i=0}^{\nhalf}a_i^T(q)2^i 
\frac{\alpha_{n,k+1,i}}{\alpha_{n,k+1,0}}.  \mbox{ Thus, }\\
\odtr_k(\sL_q^T) - \odtr_{k+1}(\sL_q^T) & = & 
\sum_{i=0}^{\nhalf} a_i^T(q)2^i \Bigg( 
\frac{\alpha_{n,k,i}}{\alpha_{n,k,0}} - 
\frac{\alpha_{n,k+1,i}}{\alpha_{n,k+1,0}}
\Bigg)\\
& = & 
\sum_{i=1}^{\nhalf} a_i^T(q)2^i \Bigg( 
\frac{\alpha_{n,k,i}}{\alpha_{n,k,0}} - 
\frac{\alpha_{n,k+1,i}}{\alpha_{n,k+1,0}}
\Bigg).
\end{eqnarray*}

As mentioned earlier, when $i\geq 1$ the polynomial 
$a_i^T(q)$ is a polynomial
in $q^2$ with non negative coefficients and so the term 
$2^ia_i^T(q)$ is non negative for all
$q \in \RR$ and $i \geq 1$.  Combining with Lemma 
\ref{lem:main_ineq}, we get that each term in the summation is non negative,
completing the proof.
{\hspace*{\fill}{\eod}}

Recall the poset $\GTS_n$ mentioned in Section \ref{sec:intro}.
Nagar and Sivasubramanian in \cite[Lemma 23]{mukesh-siva-imm-poly} 
showed that going up along $\GTS_n$ poset 
weakly decreases $a_i^T(q)$ for each $i$ and for all $q\in \RR$ and
hence weakly
decreases $\od_{\lambda}(\sL_q^T)$ for each $\lambda \vdash n$.  
By combining  
\cite[Lemma 23]{mukesh-siva-imm-poly} 
with Theorem  \ref{thm:two_row_imm_ineqs} we get the following.
\begin{corollary}
\label{cor:gts_poset}
Consider the $\GTS_n$ poset on trees with $n\geq 5$ vertices.  
Let $T_1, T_2$ be trees
with $T_2$ covering $T_1$ in $\GTS_n$. Then, for all $q\in\RR$  and 
for $k=1,2\ldots, \nhalf$,   we have 
\begin{enumerate}
	\item  
$\odtr_{k-1}(\sL_q^{T_1}) \geq \odtr_{k}(\sL_q^{T_1}) \geq \odtr_{k}(\sL_q^{T_2}).$
\item	$	\odtr_{k-1}(\sL_q^{T_1}) \geq \odtr_{k-1}(\sL_q^{T_2}) \geq \odtr_{k}(\sL_q^{T_2}).$
\end{enumerate}
\end{corollary}
\begin{proof}
We sketch a proof of (1) above.  The proof of (2) is very similar
and hence is omitted. 
Theorem \ref{thm:two_row_imm_ineqs} gives us
$\odtr_{k-1}(\sL_q^{T_1}) \geq \odtr_{k}(\sL_q^{T_1})$.  When
the shape is the same and $T_2$ covers $T_1$ in $\GTS_n$, then
by \cite[Lemma 23]{mukesh-siva-imm-poly}, 
going up along $\GTS_n$ poset weakly decreases $a_i^T(q)$ for 
each $i$ and for all $q\in \RR$.  Arguing as in the proof
of Theorem \ref{thm:two_row_imm_ineqs} gives us
$\odtr_{k}(\sL_q^{T_1}) \geq \odtr_{k}(\sL_q^{T_2}),$ completing
the proof.
\end{proof}

\section{Polynomials and successive differences}
\label{sec:polynomials}

Recall the tables containing $\alpha_{n,k,i}$'s for $n=7,8$ 
given in Example \ref{eg:recursive-from-prev}.
When $n =  2\ell$, in Lemma \ref{lem:ank:rec:1},
we saw a relation between the successive difference of 
coefficients of the polynomial $(1+x+x^2)^{\ell}$ 
and $\alpha_{2\ell,k,\ell}$'s (that is, the entries of the last row).  
We next give similar identities for other $\alpha_{n,k,i}$'s 
(that is, for entries of other rows of the table).  
We need to define a sequence of polynomials.  Since the definition
depends on the parity of $n$, we define the two sequences separately and
then give their properties.

\subsection{When $n=2\ell$ is even}
\label{subsec:poly_seq_n_even}

For $0 \leq i \leq \ell$, define the sequence of polynomials 
$p_{2\ell,i}(x) = (1+x)^{2\ell-2i}(1+x+x^2)^i = \sum_{k=0}^{2\ell} 
p_{2\ell,i,k}x^k$.  For example,
when $2\ell=8$, we tabulate the polynomials 
below:

\vspace{2 mm}

$
\begin{array}{|r|l|} \hline
p_{8,0}(x) & (1+x)^8 \\ \hline
p_{8,1}(x) & (1+x)^6 \times (1+x+x^2) \\ \hline
p_{8,2}(x) & (1+x)^4 \times (1+x+x^2)^2 \\ \hline
p_{8,3}(x) & (1+x)^2 \times (1+x+x^2)^3 \\ \hline
p_{8,4}(x) & (1+x+x^2)^4 \\ \hline
\end{array}
$
\vspace{2 mm}

It is easy to see that $p_{8,1}(x) = 
1 + 7x + 22x^2 + 41 x^3 + 50 x^4 + 41x^5 + 22x^6 + 7x^7 + x^8$.
Taking successive differences of coefficients, we get $1=1-0$, $6=7-1$,
$15=22-7$, $19=41-22$, $9 = 50-41$ and thus, we get the row corresponding
to $i=1$ in the table for $n=8$.  Similarly, from 
$p_{8,2}(x)$, we get the row of the table for $n=8$, corresponding to $i=2$.


\subsection{When $n=2\ell+1$ is odd}
\label{subsec:poly_seq_n_odd}
For $0 \leq i \leq \ell$, define the sequence of polynomials 
$p_{2\ell+1,i}(x) = (1+x)^{2\ell+1-2i}(1+x+x^2)^i = \sum_{k=0}^{2\ell+1} 
p_{2\ell+1,i,k}x^k$.  For example,
when $2\ell+1=7$, we tabulate the polynomials 
below:

\vspace{2 mm}

$
\begin{array}{|r|l|} \hline
p_{7,0}(x) & (1+x)^7 \\ \hline
p_{7,1}(x) & (1+x)^5 \times (1+x+x^2) \\ \hline
p_{7,2}(x) & (1+x)^3 \times (1+x+x^2)^2 \\ \hline
p_{7,3}(x) & (1+x) \times (1+x+x^2)^3 \\ \hline
\end{array}
$

One can check that $p_{7,3}(x) = 1+4x+9x^2 + 13x^3 + 13x^4 + 
9x^5 + 4 x^6 + x^7$.  From this polynomial, taking sucessive 
difference as done in Subsection \ref{subsec:poly_seq_n_even},
we get 1,3,5,4 which is the row corresponding to $i=3$
in the table when $n=7$.  One can 
check that other rows are obtained  in a similar manner.

\subsection{Successive differences}
We combine the two definitions of the polynomials
as follows.
For $0\leq i \leq \ell$ and for $s\in \{0,1\}$, define the sequence 
of polynomials:
$$p_{2\ell+s,i}(x)=(1+x)^{2\ell -2i+s}(1+x+x^2)^i
=\sum_{k=0}^{2\ell+s}p_{2\ell+s,i,k}x^k.$$
In the following lemma, we give a similar {\it successive difference}
interpretation for other $\alpha_{n,k,i}$'s.  

\begin{lemma}
\label{lem:gen_poly}
With the notation above, when $\ell \geq 2$ and $1 \leq i \leq \ell$, 
we have the following.
 \begin{enumerate}
\item \label{1} $p_{2\ell+s,i,k}=p_{2\ell+s-2,i-1,k} + 
p_{2\ell+s-2,i-1,k-1} + p_{2\ell+s-2, i-1,k-2}.$
\item \label{2} $\alpha_{2\ell+s,k,i}=p_{2\ell+s,i,k}-
p_{2l+s,i,k-1}.$
 \end{enumerate}
 \end{lemma}
\begin{proof}
\begin{enumerate}
   \item \label{1.} By definition, we have  
   \begin{align*}
   p_{2\ell+s,i,k} & = \mbox{ Coeff. of } x^k \mbox{ in } 
					(1+x)^{2\ell-2i+s}(1+x+x^2)^{i}\\
   & = \mbox{ Coeff. of } x^k \mbox{ in } (1+x)^{2(\ell-i+1)+s-2}(1+x+x^2)^{i-1}(1+x+x^2)\\
   & = \mbox{ Coeff. of } x^k \mbox{ in } p_{2\ell+s-2,i-1}(x)(1+x+x^2)\\
   & = p_{2\ell+s-2,i-1,k}+p_{2\ell+s-2,i-1,k-1}+p_{2\ell+s-2,i-1,k-2}.
   \end{align*}

\item \label{2.} To prove the second part, we use induction on
$n=2\ell+s$ as done in the proof of Lemma \ref{lem:ank:rec:1}. It 
is easy to verify the statement for small values of $\ell,k$ and $i$.  
Assume that the lemma is true for all 
values less than $n=2\ell+s$.  We use the
Murnaghan Nakayama lemma 
to prove the lemma when $n=2\ell+s$. From the definition of
$\alpha_{n,k,i}$ we have
\begin{align*}
2^i\alpha_{2\ell+s,k,i} =&  \sum_{j=0}^{i}\chitr_{2\ell+s,k}(j)\binom{i}{j}\\
=& \> 2\left[
\sum_{j=0}^{i-1}\chitr_{2\ell+s-2,k}(j)\binom{i-1}{j}+
\sum_{j=0}^{i-1}\chitr_{2\ell+s-2,k-1}(j)\binom{i-1}{j} \right. \\
& +\left. \sum_{j=0}^{i-1}\chi_{2\ell+s-2,k-2}(j)\binom{i-1}{j} \right]\\
=& \> 2\left[2^{i-1}\alpha_{2\ell+s-2,k,i-1}+ 
2^{i-1}\alpha_{2\ell+s-2,k-1,i-1}+2^{i-1}\alpha_{2\ell+s-2,k-2,i-1}\right]
\end{align*}
By induction, we get 
\begin{align*}
\alpha_{2\ell+s,k,i}  =& \> \alpha_{2\ell+s-2,k,i-1}+
\alpha_{2\ell+s-2,k-1,i-1}+\alpha_{2\ell+s-2,k-2,i-1}\\
 =& \> p_{2\ell+s-2,i-1,k}-p_{2\ell+s-2,i-1,k-1}+ p_{2\ell+s-2,i-1,k-1}-p_{2\ell+s-2,i-1,k-2}\\
& \> +p_{2\ell+s-2,i-1,k-2}-p_{2\ell+s-2,i-1,k-3}\\
=& \> p_{2\ell+s,i,k}-p_{2\ell+s,i,k-1}.
\end{align*}
\end{enumerate}
where the last equality follows from first part.  The proof is complete.
\end{proof}

\section{A probabilistic interpretation}
\label{sec:prob_interpret}

In this section, we give a path based interpretation and recast 
Lemma \ref{lem:main_ineq} in a probabilistic setting.  Recall
from Remark \ref{rem:dim_i-zero-row}, that 
$\alpha_{n,k,0} = \chitr_{n,k}(\id)$, where $\chitr_{n,k}(\id)$ is the
dimension of the irreducible representation of $\SSS_n$ indexed by
the two-row partition $\TwoRow_k = n-k,k$, which by the 
Hook-length formula equals the number of Standard Young 
Tableaux (SYT henceforth) of shape $n-k,k$.  


Consider non negative lattice paths on the plane from $(0,0)$ to 
$(n,n-2k)$, consisting of $n-k$ Up steps which go from $(x,y)$ to $(x+1,y+1)$ 
denoted U and $k$ Down steps which go from $(x,y)$ to $(x+1,y-1)$,
denoted D, that stay on or above the $x$-axis.  By definition, all 
paths we consider stay on or to the right of the $y$-axis and a 
lattice path is termed {\it non negative} if it stays on or above
the $x$-axis.
For $k \leq \nhalf$, let $\NLP(n,n-2k)$  be the set of 
such non negative lattice paths from $(0,0)$ to $(n,n-2k)$.  
Since $\NLP(2n,0)$ is the set of Dyck paths of length $2n$ (or semi length
$n$), we refer to $\NLP(n,n-2k)$ as {\it Generalized Dyck 
paths} with the word {\it generalized} implying that the number of
Up steps is larger than the number of Down steps.

\begin{remark}
\label{rem:syt-nlp-bijection}
The following well known bijection 
maps Standard Young tableaux of shape $n-k,k$ to 
$\NLP(n,n-2k)$ as follows.  Given
an SYT $T$ of shape $n-k,k$, consider the path $P_T$ whose 
$i$-th step is U if $i$ is in the first row of $T$ and
whose $i$-th step is D if $i$ is in the second row of $T$.
\end{remark}

Combining Remark \ref{rem:dim_i-zero-row} with the bijection in
Remark \ref{rem:syt-nlp-bijection}, the denominator terms in 
Lemma \ref{lem:main_ineq}, 
are the cardinalities of $\NLP(n,n-2k)$ and $\NLP(n,n-2k-2)$ 
respectively.   Our first aim is to give a similar interpretation for
$\alpha_{n,k,i}$ as the cardinality of some set of paths.  
Our interpretation will depend on the parity of $n$.  Let
\begin{equation}
\label{eqn:qnk_defn}
(1+x+x^2)^n =\sum_{k=0}^{2n} p_{n,k}x^k 
\hspace{6 mm} \mbox{ and } \hspace{6 mm}
(1+x+1/x)^n = \sum_{k=-n}^n q_{n,k}x^k.
\end{equation}
As $1+x+x^2 = x(1 + x + 1/x)$,  when $-n \leq k \leq n$, we get 
$p_{n,n+k} = q_{n,k}$.  Thus, the $p_{n,k}$'s are translates
of the $q_{n,k}$'s.  It is also easy to see that 
$(1+x+x^2)^n$ is a palindromic polynomial of degree $2n$.
That is, for $0 \leq k \leq 2n,$  we have $p_{n,k} = p_{n,2n-k}$. 

From \eqref{eqn:qnk_defn}, a moments reflection gives the
following interpretation for $q_{n,k}$: $q_{n,k}$ equals 
the number of lattice paths from $(0,0)$ to
$(n,k)$ where we are allowed the following three types 
of steps: $U$ from $(x,y)$ to $(x+1,y+1)$, $H$ from $(x,y)$ to 
$(x+1,y)$ and $D$ from $(x,y)$ to $(x+1,y-1).$
Note that these lattice paths need not be non negative.
We call such paths as UHD paths.  By translating, we can get 
a path based interpretation
for the $p_{n,k}$'s.  We now bifurcate our discussion
into two parts depending on the parity of $n$.

\subsection{When $n=2\ell$ is even}
When $n=2\ell$, for $0 \leq k \leq \ell$, by Lemma \ref{lem:ank:rec:1}, 
we have $\alpha_{2\ell,k,\ell} = p_{\ell,k} - p_{\ell,k-1}$ 
(where $p_{\ell,-1} = 0$).  
Callan in \cite{riordan-central-trinomial} 
showed that 
the difference between the central trinomial coefficient and its 
predecessor is the Riordan number $R_n$ which counts 
the number of non negative UHD paths from $(0,0)$ to 
$(n,0)$ with no $H$ steps at height 0.  Here {\sl non negative
UHD paths} are UHD paths which do not go below the $x$-axis.  
By Callans result, 
we get that $\alpha_{2\ell,\ell,\ell} = R_{\ell} = 
p_{\ell,\ell} - p_{\ell,\ell-1}$.  
We will need {\it Generalized Riordan paths} which are defined 
as non negative UHD paths with no $H$ step
at height 0, but are from $(0,0)$ to $(n,k)$
where $k$ need not be zero.
Callan's result is actually more general and 
gives an interpretation for the numbers
$\alpha_{2\ell, k,\ell}$ (which we had denoted 
as $\last_{\ell,k}$) 
as the cardinality of a set of {\it Generalized Riordan paths}.
We give a proof for completeness.

Recall for $0\leq k \leq \ell,$ that $q_{\ell,k}$, the 
coefficient of $x^k$ in $(x^{-1}+1+x)^{\ell}$ is the number of 
UHD paths from $(0,0)$ to $(\ell,k)$. 
By translation, for $0\leq k \leq \ell$, 
$p_{\ell,\ell+k}$ is the number of UHD paths from $(0,0)$ to 
$(\ell,k)$. 
Let $\UHD(\ell,\ell-k)$ be the set of UHD paths from $(0,0)$ 
to $(\ell,\ell-k)$. Since $p_{\ell,k}=p_{\ell,2\ell-k}=|\UHD(\ell,\ell-k)|$, 
we will give a combinatorial proof that 
$\alpha_{2\ell,k,\ell} = p_{\ell,k}-p_{\ell,k-1}=
|\UHD(\ell,\ell-k)|-|\UHD(\ell,\ell-k+1)|$.

\begin{lemma}[Callan]
\label{lem:path-interpretn-trinomial-diff}
Let $\ell$ be a positive integer and let $k \leq \ell$ be a 
non negative integer.  The number of Generalized Riordan paths 
from $(0,0)$ to $(\ell,k)$  
equals $p_{\ell,\ell-k} - p_{\ell,\ell-k-1}$.
That is,
$\alpha_{2\ell,k,\ell}=|\UHD(\ell,\ell-k)|-|\UHD(\ell,\ell-k+1)|$.
Thus, $\alpha_{2\ell,k,\ell}$ equals the number of
Generalized Riordan paths from $(0,0)$ to $(\ell,\ell-k)$.
\end{lemma}
\begin{proof}
For $0 \leq k \leq \ell$
let $\RTP(\ell,\ell-k)$ denote the set of Generalized Riordan paths
from $(0,0)$ to $(\ell,\ell-k)$ without a horizontal step at height 
zero. We will prove the Lemma by giving a bijection $f$ from
the set $\UHD(\ell,\ell-k)/\RTP(\ell,\ell-k)$ 
to the set $\UHD(\ell,\ell-k+1)$.
		
	Suppose $P\in\UHD(\ell,\ell-k)$ has either a $H$ step at 
ground level or dips strictly below the $x$ axis at some point or both.  
Denote by $R$ the subpath of $P$ starting from the $x$-axis 
after either the last horizontal step at height 0, or after the last 
time $P$ went below the $x$-axis, (if both events happen, choose whichever 
event happens later).   Thus, 
$R$ is the longest Generalized Riordan sub-path that starts somewhere 
on the $x$-axis and ends $P$.   Consider the step $X$ in $P$ that precedes 
$R$.   It is easy to check that $X$ cannot be $D$. Thus, we have 
two cases based on $X$.
		
{\bf Case 1 (when $X = H$):}  
We have $P=SXR$ for some sub-path $S$ of $P$ ending somewhere on the 
$x$-axis. Define $f(P)$ as follows: $f(SXR)= \overline{S}UR^{+1},$
where  $R^{+1}$  is sub-path obtained by shifting $R$ from the  
ground level to level 1 and $\overline{S}$ is obtained by flipping 
$S$ with respect to the $x$ axis.
		
{\bf Case 2 (when $X=U$):}  We have
$P=SXR$ for some sub-path $S$ of $P$ ending at height $-1$. 
Define  $f(SUR)=\overline{S}HR^{+1}$, where $R^{+1}$ and 
$\overline{S}$ are as defined in Case 1.
		
We note that any path with first step $U$ gets mapped under $f$ to a path 
with first step $D$ and vice-versa.  The map $f$ sends paths whose 
first step is $H$ to paths with first step $H$ itself. 

{\bf Inverse map $f^{-1}$: }  To defined the inverse of $f$, let 
$P\in \UHD(\ell,\ell-k+1)$ be a path from $(0,0)$ to $(\ell,\ell-k+1)$. 
Let $R$ be the largest subpath of $P$ that ends $P$ and does not have  a
$H$ step at level $1$ or goes below level $1$.  
As before, let $X$ be the step in $P$ that precedes $R$.  Note that
$X$ cannot be $D$.  Thus we have the following two cases. 
		
{\bf Case 1 (when $X=H$): } We have $P=SHR$ for some subpath $S$.  
Define $f^{-1}(SHR)= \overline{S}UR^{-1},$ where 
$R^{-1}$ is sub-path obtained by shifting $R$ from the level $1$ 
to ground level and $\overline{S}$ is as defined as in the definition
of $f$.
	
{\bf Case 2 (when $X=U$): } Whe have $P=SUR$ for some subpath $S$. 
Define $f^{-1}(SUR)= \overline{S}HR^{-1}.$
		
It is easy to check that $f \odot f^{-1}= \id$, 
the identity map. The proof is complete. 
\end{proof}

\begin{remark}
The bijection $f$ defined in the proof of Lemma
\ref{lem:path-interpretn-trinomial-diff}
 is illustrated in Figure \ref{fig:Callan_proof} 
where $f(UDDUUUUH) = DUUHUUUH$  and 
$f^{-1}(UDDHDUUU) = DUUHUDDH$.
\end{remark}

\begin{figure}[h]
\centering
\includegraphics[scale=0.6]{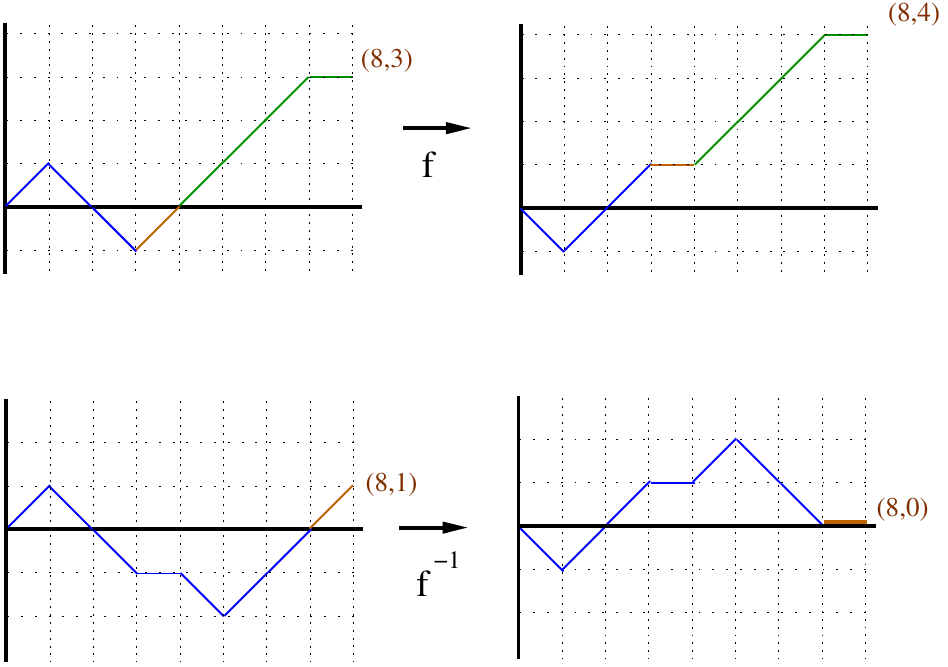}
\caption{Examples of the bijection $f$ and its inverse $f^{-1}.$}
\label{fig:Callan_proof}
\end{figure}

\begin{example}
\label{eg:illustrate_paths}
We illustrate Lemma \ref{lem:path-interpretn-trinomial-diff}
when $\ell=4$.  We clearly have $p_{4,-1}=0$, $p_{4,0} = 1$, $p_{4,1} = 4$, 
$p_{4,2} = 10$, $p_{4,3} = 16$ and $p_{4,4} = 19$.  From Example 
\ref{eg:recursive-from-prev}, we have the following table of 
$\alpha_{8,4,k}$.  Clearly, $\alpha_{8,4,k} = p_{4,k} - p_{4,k-1}$ and
we have the following sets of Generalized Riordan paths.

$ \begin{array}{|r|c|c|c|c|c|c|}  \hline
  & \lambda = 8 & \lambda = 7,1 & \lambda = 6,2 & \lambda = 5,3 & \lambda=4,4\\ \hline
  i=4 & 1 & 3 & 6 & 6 & 3 \\ \hline
\mbox{End point} & (4,4) & (4,3) & (4,2) & (4,1) & (4,0) \\ \hline
\mbox{Path sets} & \RTP(4,4) & \RTP(4,3) &  \RTP(4,2) & \RTP(4,1)&  \RTP(4,0) \\  \hline
\end{array}$

\vspace{2 mm}

where
$\RTP(4,4) = \{UUUU \}$, $\RTP(4,3) = \{UUUH, UUHU, UHUU \}$,

$\RTP(4,2) = \{UUHH,UHHU, UHUH, UUUD, UUDU, UDUU\}$,

$\RTP(4,1) = \{ UDUH, UUDH, UHDU, UHUD, UUHD, UHHH\}$ and 

$\RTP(4,0) = \{UDUD, UUDD, UHHD \}$.  The set of 
paths in $\RTP(4,2)$ are drawn in Figure \ref{fig:riordan4_2}.
\end{example}
 
\begin{figure}[h]
\centerline{\includegraphics[scale=0.65]{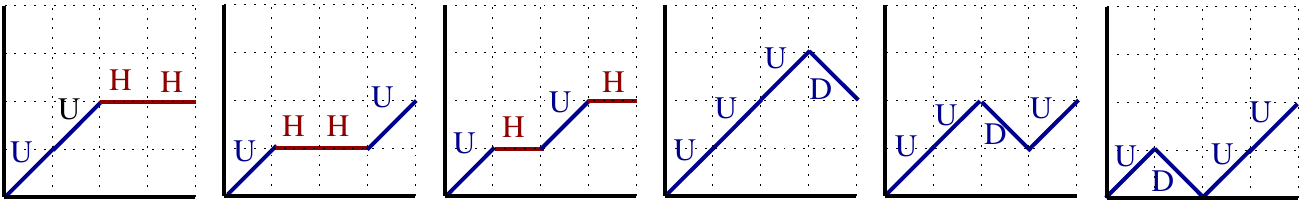}}
\caption{The set $\RTP(4,2)$ of Generalized Riordan paths from $(0,0)$ to $(4,2)$.}
\label{fig:riordan4_2}
\end{figure}

In our next lemma, we interpret Generalized Riordan paths as 
generalized Dyck paths with restrictions on the positions of 
its peaks.  We denote paths with only Up and Down steps as UD paths.
The following 
interpretation of Riordan paths is known (see OEIS) and we give a 
simple proof as we need a version for Generalized Riordan paths as well.  
Given a UD path $P$, a peak is a lattice point $(p,q)$ on $P$ such that
an up-step ends at $(p,q)$ and a down-step starts at $(p,q)$.

\begin{lemma}
\label{lem:riordan-dyck}
Let $n = 2\ell$.  For $0 \leq k \leq \ell$, there is a 
bijection $f$ from $\RTP(\ell,\ell-k)$ to the set 
$\NLP(2\ell,2\ell-2k)$ of generalized Dyck paths from 
$(0,0)$ to $(2\ell,2\ell-2k)$ that have 
$2\ell-k$ Up steps, $k$ Down steps and have no peaks at any
odd height.  Thus, $\alpha_{2\ell, k, \ell}$ is the number of 
generalized Dyck paths from $(0,0)$ to $(2\ell, 2\ell -2k)$
with $2\ell-k$ Up steps, $k$ Down steps, that have no peaks at 
any odd height.
\end{lemma}
\begin{proof}
Let $P \in \RTP(\ell,\ell-k)$ with $P = a_1,a_2,\ldots,a_{\ell}$ be a 
Generalized Riordan path from $(0,0)$ to
$(\ell, \ell-k)$ where $a_i = U/H/D$, 
depending on the type of the
$i$-th step of $P$.  Perform the following operations: 
change $U$ to $U,U$, change $D$ to $D,D$ and change $H$ to $D,U$.  
This will convert $P$ to 
$f(P) = Q = b_1, b_2, \ldots, b_{2\ell}$ where $Q$ is an UD 
path.  We note the following properties of the bijection $f$.

{\bf (Property 1) $Q$ is a non negative path : } As 
$P \in \RTP(\ell, \ell-k)$ 
and thus has no horizontal steps 
at height 0.  Thus, changing a $H$ step in $P$ to $D,U$ in $f(P)$ will
not make the path $f(P)$ go below height 0.  Further, since $P$ 
is non negative, any $D$ step in $P$ is preceded by a $U$ step prior
to it.  This ensures that while changing $D$ in $P$ to $D,D$ in $f(P)$ 
we would have earlier changed a $U$ in $P$ to a $U,U$ in $f(P)$ 
and hence this change will also not make $f(P)$ go below height 0.

{\bf (Property 2) $Q$ has no peaks at any odd height :} To see
this, note that a peak will occur in $f(P)$ iff there is a consecutive 
$U,D$ pair.  Suppose $(b_i,b_{i+1}) = (U,D)$, then it is easy to 
see that $i$ is even.  As $i$ is even, this means that the height 
at which the peak occurs in $f(P)$ is also even. Thus any peak
of $f(P)$ only occurs at an even height.
The proof is now complete.
\end{proof}

\begin{figure}[h]
\centerline{\includegraphics[scale=0.55]{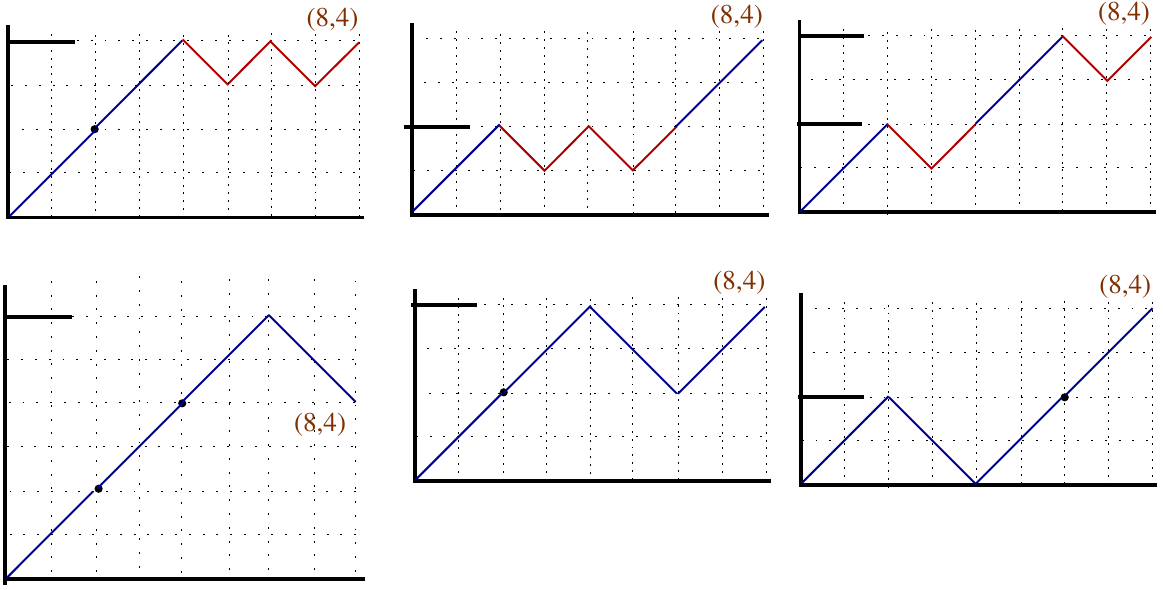}}
\caption{The generalized Dyck paths obtained under the 
bijection $f$ applied to paths in $\RTP(4,2)$.}
\label{fig:riordan4_2-up-down}
\end{figure}

Figure \ref{fig:riordan4_2-up-down} shows the generalized Dyck paths
output by the bijection $f$ described in Lemma \ref{lem:riordan-dyck}
on paths $P \in \RTP(4,2)$.
Note that Lemma \ref{lem:riordan-dyck} gives an interpretation for 
$\alpha_{2\ell, k, \ell}$, that is for entries in the last 
row when $n = 2\ell$.  Using this as a building block, we give another
expression for $\alpha_{2\ell, k, i}$ in terms of $\alpha_{2m,k,m}$.
This will enable us to give 
an interpretation for $\alpha_{2\ell, k, i}$.

\begin{lemma}
\label{lem:binomial-riordan}
Let $n = 2\ell$ and let $0 \leq k,i \leq \ell$.  Then, 
$$\alpha_{2\ell, k, i} = 
\sum_{t=0}^{\ell-i} \binom{\ell-i}{t} \alpha_{2\ell-2t,k-t,\ell-t}.$$
\end{lemma} 
\begin{proof}
By Lemma \ref{lem:gen_poly}, $\alpha_{2\ell, k,i}$ is the 
difference of succesive coefficients from the polynomial
$p_{2\ell, i}(x)$.  Set $b = 1+x+x^2$.  Then,
for $k \geq 0$, we clearly have $p_{2k,k}(x) = b^k$.  It is 
further clear that 
\begin{eqnarray*}
p_{2\ell, i}(x) & = & (x+b)^{\ell-i} p_{2i,i}(x) = (x+b)^{\ell-i}b^i\\
 & = & 
\sum_{t=0}^{\ell-i} \binom{\ell-i}{t}x^t b^{\ell-t}
= 
\sum_{t=0}^{\ell-i} \binom{\ell-i}{t}x^t p_{2\ell-2t,\ell-t}(x)
\end{eqnarray*}
As taking the difference of successive coefficients 
is a linear operator, we get the desired equation, completing
the proof.
\end{proof}

\begin{example}
We illustrate Lemma \ref{lem:binomial-riordan} by getting
the last column of the table when $n=8$. The following
data can be easily verified. 
\vspace{2 mm}

\begin{tabular}{|l|r|r|r|r|r|} \hline
$\ell$ & 0 & 1 & 2 & 3 & 4 \\ \hline
$\alpha_{2\ell,\ell,\ell}$ & 1 &  0 &  1 &  1 &  3 \\ \hline
\end{tabular}

\vspace{2 mm}

From the table for $n=8$ in Example \ref{eg:recursive-from-prev}, 
one can easily verify the construction of the entries in the last column 
using the elements $\alpha_{2\ell,\ell,\ell}$ as
follows.
\begin{tabbing}
asdsaa \= as \= asa \= as \=  \kill
$\alpha_{8,4,3}$ \> $=$ \> $4$ \> $=$ \> $\alpha_{8,4,4} + \alpha_{6,3,3},$\\
$\alpha_{8,4,2}$ \> $=$ \> $6$ \> $=$ \> $\alpha_{8,4,4} + 2 \alpha_{6,3,3} + \alpha_{4,2,2},$\\
$\alpha_{8,4,1}$ \> $=$ \> $9$ \> $=$ \> $\alpha_{8,4,4} + 3 \alpha_{6,3,3} + 3 
\alpha_{4,2,2} + \alpha_{2,1,1},$\\
$\alpha_{8,4,0}$ \> $=$ \> $14$ \> $=$ \> $\alpha_{8,4,4} + 4 \alpha_{6,3,3} + 6 
\alpha_{4,2,2} + 4 \alpha_{2,1,1} + \alpha_{0,0,0}.$\\
\end{tabbing}
\end{example}

\vspace{-4 mm}

Using Lemma \ref{lem:binomial-riordan},  we give an 
interpretation for the numbers 
$\alpha_{2\ell,k,i}$  when $i < \ell$.  It will again be the 
the cardinality of a set of generalized Dyck paths with odd peaks occurring
at restricted positions.  All our generalized Dyck paths will be from 
$(0,0)$ to $(2\ell, 2\ell-2k)$.  Divide the $2\ell$ steps on the 
$x$-axis into $\ell$ intervals of length 2 each.  Thus, we have 
intervals $s_1 = (0,2)$, $s_2 = (2,4), \ldots, 
s_{\ell} = (2\ell-2, 2\ell)$.   
For $0 \leq i \leq \ell$, define the sets 
$D_i = \{1,2,\ldots,i\}$.  Thus $D_0 = \emptyset$, $D_1 = \{1\},
D_2 = \{1,2\}$ and so on.  We will permit peaks to have an 
odd height at a point $(x,y)$ where $x \in D_i$.

\begin{lemma}
\label{lem:convol_binom_type}
With the notation described above, $\alpha_{2\ell, k, i}$ is
the cardinality of the set 
of generalized Dyck paths from $(0,0)$ to $(2\ell, 2\ell-2k)$ with 
$2\ell-k$ Up steps, $k$ Down steps and odd
peaks contained in the set $D_{\ell-i}$.
\end{lemma}
\begin{proof}
Our proof is inspired by the proof of Lemma 
\ref{lem:binomial-riordan}.  We construct generalized Dyck 
paths with peaks at odd height in the set $D_{\ell-i}$ as follows.  
If there are peaks at odd heights, then as done in the proof 
of Lemma \ref{lem:riordan-dyck}, it is clear that any such peak
will occur at position $(x,y)$ where both $x,y$ are odd positive
integers.  Thus, such an odd peak  causing ``$U,D$" pair of 
steps has to be 
in positions indexed by $s_d$ for some $d \in \{1,2,\ldots,\ell\}$.

We claim that the number of generalized Dyck paths with $t$ odd height peaks 
in the set $D_{\ell-i}$ is $\binom{\ell-i}{t} \alpha_{2\ell-2t,k-t,\ell-t}$.
If such a path $P$ is written as a string of $U,D$'s, any peak 
will have a 
consecutive ``$U,D$" substring.  Note that if $P$ has $t$ 
peaks at odd heights, then removing the $t$ ``$U,D$" pairs
will give a generalized Dyck path $Q$ of length $2\ell-2t$ with 
no change in the final height of the path (thus having 
$2\ell-k-t$ Up and $k-t$ Down steps) and with $t$ fewer $U$ 
steps and $t$ fewer $D$ steps.  Further, $Q$ has no odd peaks.
This argument goes both ways. 

Given a generalized Dyck path $Q$ with a total of $2\ell-2t$ steps from 
$(0,0)$ to $(2\ell-2t,2\ell-2k)$
that has $2\ell-k-t$ Up steps and $k-t$ Down steps 
with no odd peaks, one can choose a subset $T$ of size $t$ 
from $D_{\ell-i}$  in $\binom{\ell-i}{t}$ ways and insert a 
``$U,D$" pair at position $s_d$ for $d \in T$.  This completes 
the proof.
\end{proof}

\begin{example}
We illustrate the bijection described in Lemma 
\ref{lem:convol_binom_type} to get $\alpha_{8,4,0}$. 
We thus need Dyck paths from $(0,0)$ to $(8,0)$.
Since we do not change the height, our building blocks are 
Dyck paths without peaks at odd heights from $(0,0)$ to $(2m,0)$ 
for non negative integers $m$.  These are given in Figure 
\ref{fig:catalan-4-base} with different colours for added
clarity.  The set of paths formed is given in Figure 
\ref{fig:catalan-4-all} where the same colours are used and 
odd peak causing ``U,D" pairs are drawn using dotted lines.  
\begin{figure}[h]
\centerline{\includegraphics[scale=0.55]{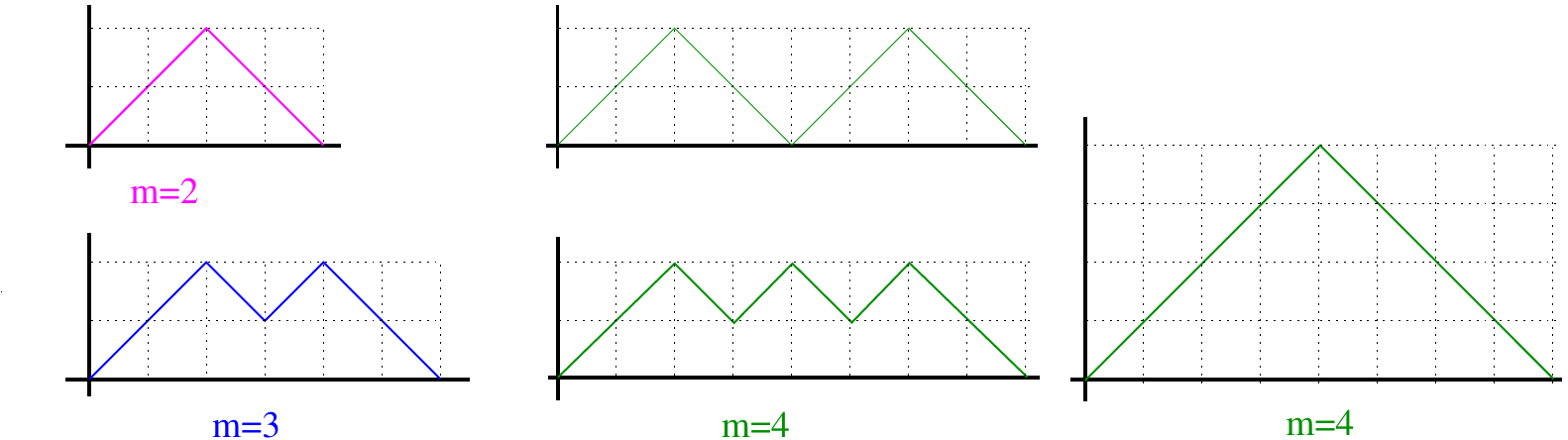}}
\caption{Up down paths with no peaks at an odd height.}
\label{fig:catalan-4-base}
\end{figure}

\begin{figure}[h]
\centerline{\includegraphics[scale=0.45]{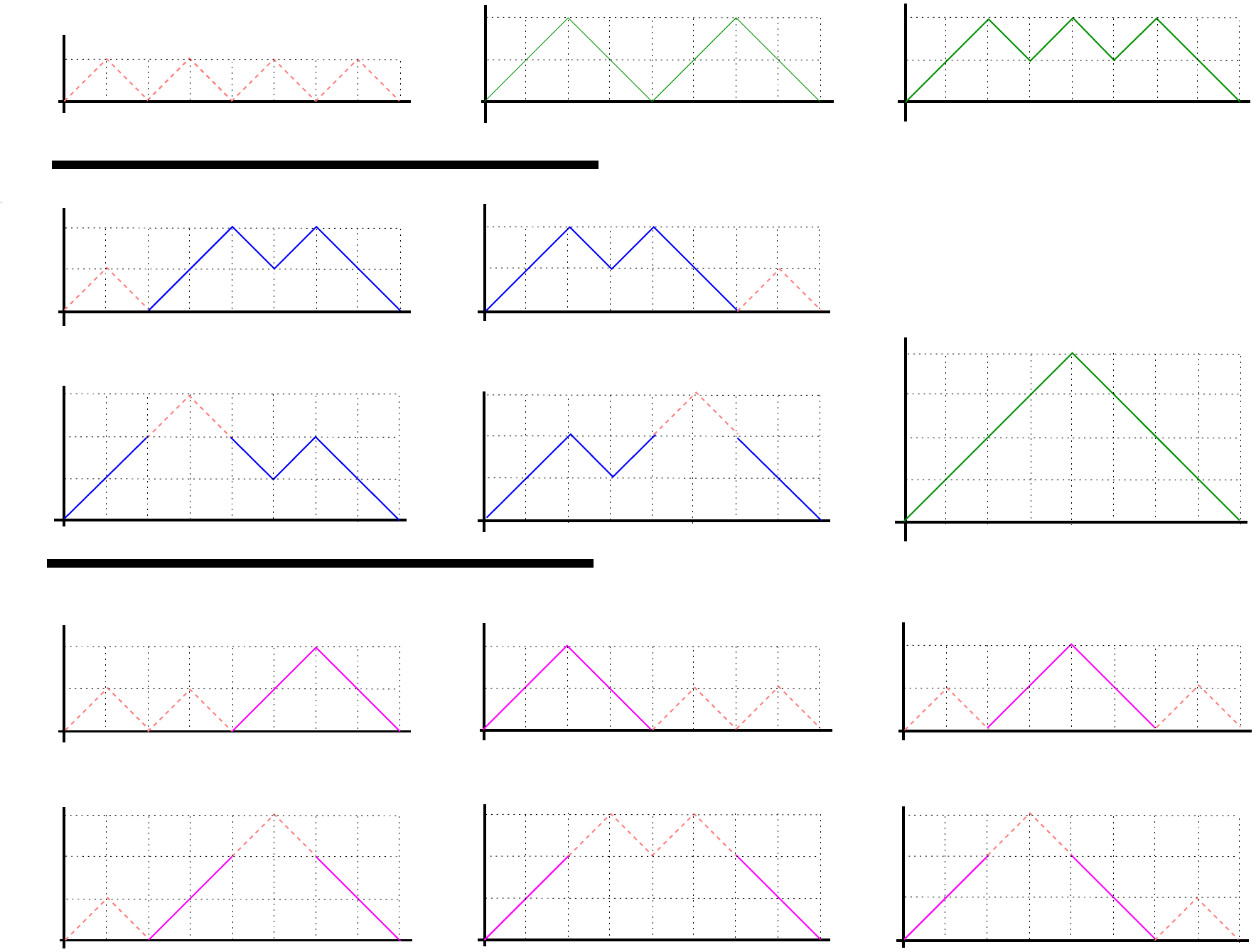}}
\caption{$\alpha_{8,4,0}$ counts all $14$ Catalan paths of semi-length 
$4$ that have peaks at odd heights in $D_4$.}
\label{fig:catalan-4-all}
\end{figure}
\end{example}

\begin{remark}
\label{rem:interpret_path_length}
Note that when $i = \ell$, 
Lemma \ref{lem:convol_binom_type} 
gives Lemma \ref{lem:riordan-dyck}.  Further, $\alpha_{2\ell,k,i}$ 
is a term that occurs in the normalized immanant computation of 
the partition $TwoRow_k$.  Thus, when $\alpha_{n,k,i}$ is viewed
as the cardinality of a restricted set of generalized Dyck paths as 
given in Lemma \ref{lem:convol_binom_type}, the second parameter 
$k$ in the subscript of $\alpha_{2\ell,k,i}$ is the number of 
Down steps while the first parameter is the total number of 
steps in the generalized Dyck path. 

Recall that $\alpha_{2\ell,\ell,0} = C_{\ell}$,
the $\ell$-th Catalan number and as 
$\alpha_{2\ell-2t,\ell-t,\ell-t} = R_{\ell-t}$, where $R_{\ell}$ is
the $\ell$-th Riordan number. When $n=2\ell, k=\ell$ and $i=0$, Lemma
\ref{lem:binomial-riordan} gives us the known fact that
$C_{\ell} = \sum_{t=0}^{\ell} \binom{\ell}{t} R_{\ell-t}$.
\end{remark}

\subsection{When $n=2\ell+1$ is odd}
When $n=2\ell+1$, we use Lemma \ref{lem:murn-naka-based-recur}
which states that $\alpha_{2\ell+1,k,i} = 
\alpha_{2\ell,k,i} + \alpha_{2\ell, k-1,i}$.  
By Remark \ref{rem:interpret_path_length},
$\alpha_{2\ell,k,i}$ and $\alpha_{2\ell, k-1,i}$ are 
the cardinalities of non negative UD paths where the first parameter
$2\ell$ is the total number of steps while the second
parameter $k$ is the number of Down steps.  Further, these have 
odd height peaks in the set $D_{\ell-i}$.

The same interpretation works when $n=2\ell+1$.  Consider non negative 
UD paths with $2\ell+1$ steps containing $k$ Down steps, 
which have one more Up step as compared to paths counted by the 
set with cardinality $\alpha_{2\ell,k,i}$.  It is 
simple to see that there is a bijection between a non negative 
UD path $P$ counted by $\alpha_{2\ell,k,i}$ and the 
path $P' = P,U$ where
we append an Up step at the end of $P$.  That is, if the last step
of a path counted by $\alpha_{2\ell+1,k,i}$ is an Up step, then
by deleting it, we get an non negative UD path counted by
$\alpha_{2\ell,k,i}$.  However, if the last step is a Down step, 
then after deletion of this, we get a non negative path counted by the 
$\alpha_{2\ell,k-1,i}$.  We only need to check that appending
a Down step at the end does not create a valley at an odd height.  But
this follows from the fact that a path with $2 \ell$ steps and $k-1$
Down steps ends at a point $(2\ell,2\ell-2k+2)$ and so ends 
at a point with even $y$ co-ordinate.  Adding a Down step to 
such a path may create a peak but only at an even height.  Thus,
the set of odd height peaks after addition of a Down step at the end does 
not change.  Thus we get the following counterpart of Lemma
\ref{lem:convol_binom_type}.

\begin{lemma}
\label{lem:convol_binom_type_odd}
With the notation described above, $\alpha_{2\ell+1, k, i}$ is
the cardinality of the set 
of non negative UD paths from $(0,0)$ to $(2\ell+1, 2\ell-2k+1)$ with 
$2\ell-k+1$ Up steps, $k$ Down steps and odd
peaks contained in the set $D_{\ell-i}$.
\end{lemma}

\subsection{Probabilistic Interpretation of Lemma \ref{lem:main_ineq}}
From Lemma  \ref{lem:convol_binom_type} and Lemma
\ref{lem:convol_binom_type_odd}, we get the following 
probabilistic interpretation 
of Lemma \ref{lem:main_ineq} whose straightforward proof 
we omit.

\begin{lemma}
\label{lem:prob_interpret_paths}
Fix a positive integer $n$ and $i \leq \nmhalf$.  Then, the probability 
of a non negative UD path with $n$ total steps and with odd height
peaks contained  in the set $D_{\nhalf-i}$ decreases as the number $k$
of Down steps increases.
\end{lemma}

Recall that Remark \ref{rem:syt-nlp-bijection} gives a bijection 
between non negative UD paths with 
$n$ steps and with $k$ down steps and Standard Young Tableaux of 
shape $n-k,k$ (denoted $\SYT(n-k,k)$), we can recast 
Lemma \ref{lem:prob_interpret_paths}
in terms of SYTs.  We translate the notion of peaks at odd heights
to tableaux.  For $T \in \SYT(n-k,k)$ define position $i$ to
be a peak
if $i$ appears in the first row and $i+1$ appears in the second
row.  This is precisely saying that $i \in \DES(T)$ where
$\DES(T)$ is the descent set of $T$, which is a well studied
statistic (see the book by Stanley 
\cite[Chapter 7]{EC2}).  For a descent $i$, to get the height of its 
peak under this mapping, consider $T_{|i}$, the restriction of $T$
to the entries $\{1,2,\ldots,i\}$.  Note that $T_{|i}$ is also
an SYT.  Let $T_{|i} = a_i,b_i$ where $a_i$ and $b_i$ are the number of 
elements in the first and second row of $T_{|i}$ respectively.
Define $\rowdiff(T_{|i}) = 
a_i - b_i$ to be the difference between the number of elements in the 
first row and the number of elements in the second row of $T_{|i}$.
Define an descent $i \in T$ to have even (or odd) height 
if $\rowdiff(T_{|i})$ is even (or odd respectively).
With these definitions, recalling the set $D_{\ell-i}$, we 
can give the SYT version of Lemma \ref{lem:prob_interpret_paths}.

\begin{lemma}
\label{lem:prob_interpret_syt}
Fix a positive integer $n$ and let $i \leq \nmhalf$.  Then, the 
probability that an SYT of shape $n-k,k$ has all its descents 
with odd height in $D_{\nhalf-i}$ decreases as the number $k$
increases (and hence the shape of $T$ changes).
\end{lemma}

By running the arguments backward, it is clear that an 
alternate proof of Lemma \ref{lem:prob_interpret_syt} will give
us an alternate proof of Theorem \ref{thm:two_row_imm_ineqs}.
Thus, it would be interesting to get an alternate proof of 
Lemma \ref{lem:prob_interpret_syt}.

\section*{Acknowledgements}
The second author would like to acknowledge SERB, Government of 
India for providing a National Postdoctoral fellowship with  
file number PDF/2018/000828. 

The last author acknowledges support from project 
SERB/F/252/2019-2020 given by the Science and 
Engineering Research Board (SERB), India.

\bibliographystyle{acm}
\bibliography{main}

\begin{thebibliography}{10}

\bibitem{bapat06-procICDM}
{\sc Bapat, R.~B.}
\newblock Resistance matrix and $q$-laplacian of a unicyclic graph.
\newblock In {\em Ramanujan Mathematical Society Lecture Notes Series, 7,
  Proceedings of ICDM 2006, Ed. R. Balakrishnan and C.E. Veni Madhavan\/}
  (2008), pp.~63--72.

\bibitem{bapat-lal-pati}
{\sc Bapat, R.~B., Lal, A.~K., and Pati, S.}
\newblock A $q$-analogue of the distance matrix of a tree.
\newblock {\em Linear Algebra and its Applications 416\/} (2006), 799--814.

\bibitem{bapat-siva-third-immanant}
{\sc Bapat, R.~B., and Sivasubramanian, S.}
\newblock {\em The {T}hird {I}mmanant of q-{L}aplacian {M}atrices of {T}rees
  and {L}aplacians of {R}egular {G}raphs}.
\newblock Springer India, 2013, pp.~33--40.

\bibitem{bass}
{\sc Bass, H.}
\newblock The {I}hara-{S}elberg {Z}eta {F}unction of a {T}ree {L}attice.
\newblock {\em International Journal of Math. 3\/} (1992), 717--797.

\bibitem{bessenrodt-coincidences-ejc}
{\sc Bessenrodt, C.}
\newblock Coincidences between {C}haracters to {H}ook {P}artitions and 2-{P}art
  {P}artitions on {F}amilies arising from 2-{R}egular {C}lasses.
\newblock {\em Electronic Journal of Combinatorics 24 (3)\/} (2017).

\bibitem{riordan-central-trinomial}
{\sc Callan, D.}
\newblock {R}iordan {N}umbers {A}re {D}ifferences of {T}rinomial
  {C}oefficients.
\newblock {\em Available at
  http://pages.stat.wisc.edu/$\sim$callan/notes/riordan/riordan.pdf\/\/}
  (2006).

\bibitem{chan-lam-binom-coeffs-char}
{\sc Chan, O., and Lam, T.~K.}
\newblock Binomial {C}oefficients and {C}haracters of the {S}ymmetric {G}roup.
\newblock {\em Technical Report 693\/} (1996), National Univ of Singapore.

\bibitem{csikvari-poset1}
{\sc Csikv{\'a}ri, P.}
\newblock On a {P}oset of {T}rees.
\newblock {\em Combinatorica 30 (2)\/} (2010), 125--137.

\bibitem{csikvari-poset2}
{\sc Csikv{\'a}ri, P.}
\newblock On a {P}oset of {T}rees {II}.
\newblock {\em Journal of Graph Theory 74\/} (2013), 81--103.

\bibitem{foata-zeilberger-bass-trams}
{\sc Foata, D., and Zeilberger, D.}
\newblock Combinatorial {P}roofs of {B}ass's {E}valuations of the
  {I}hara-{S}elberg {Z}eta function of a {G}raph.
\newblock {\em Transactions of the AMS 351\/} (1999), 2257--2274.

\bibitem{mukesh-siva-hook}
{\sc Nagar, M.~K., and Sivasubramanian, S.}
\newblock Hook immanantal and {H}adamard inequalities for $q$-{L}aplacians of
  trees.
\newblock {\em Linear Algebra and its Applications 523\/} (2017), 131--151.

\bibitem{mukesh-siva-imm-poly}
{\sc Nagar, M.~K., and Sivasubramanian, S.}
\newblock Laplacian immanantal polynomials and the gts poset on trees.
\newblock {\em Linear Algebra and Applications 561\/} (2019), 1--23.

\bibitem{sagan-book}
{\sc Sagan, B.~E.}
\newblock {\em The {S}ymmetric {G}roup: {R}epresentations, {C}ombinatorial
  {A}lgorithms, and {S}ymmetric {F}unctions}, 2nd~ed.
\newblock Springer Verlag, 2001.

\bibitem{schur-immanant-ineqs}
{\sc Schur, I.}
\newblock {\"U}ber endliche {G}ruppen und {H}ermitesche {F}ormen.
\newblock {\em Math. Z. 1\/} (1918), 184--207.

\bibitem{EC2}
{\sc Stanley, R.~P.}
\newblock {\em Enumerative Combinatorics, vol 2}.
\newblock Cambridge University Press, 2001.

\bibitem{zeil-regev-surprising-relations-slc}
{\sc Zeilberger, D., and Regev, A.}
\newblock {S}urprising {R}elations {B}etween {S}ums-{O}f-{S}quares of
  {C}haracters of the {S}ymmetric {G}roup {O}ver {T}wo-{R}owed {S}hapes and
  {O}ver {H}ook {S}hapes.
\newblock {\em S\'{e}minaire Lotharingien de Combinatoire 75, Article B75c\/}
  (2016).

\end{thebibliography}

\end{document}